\documentclass[10pt]{amsart}
\usepackage[all,cmtip]{xy}
\usepackage{amsmath, amsthm, amssymb}
\usepackage{eucal, mathrsfs, yfonts}
\usepackage{amsfonts}
\usepackage{array}
\usepackage{hyperref}
\usepackage{graphicx}
\usepackage{caption}
\usepackage{verbatim}

\newtheorem{thm}{Theorem}[section]

\newtheorem{lem}[thm]{Lemma}
\newtheorem{prop}[thm]{Proposition}

\theoremstyle{definition}

\theoremstyle{remark}

\theoremstyle{Conjecture}

\numberwithin{equation}{section}

\newtheorem*{remark}{Remark}
\begin{document}

\title{Lower bound of the asymptotic complexity of self-similar fractal graphs}

\author[Konstantinos Tsougkas]{Konstantinos Tsougkas}

\address{Konstantinos Tsougkas\\
         Department of Mathematics\\
         Uppsala university, Sweden}

\email{konstantinos.tsougkas@math.uu.se}

\address{Current address: Department of Mathematics\\
	Cornell university, Ithaca, NY, USA}

\date{\today}

\begin{abstract}
	{
		We study the asymptotic complexity constant of the sequence of approximating graphs to a fully symmetric self-similar structure on a finitely ramified fractal $K$. We show how full symmetry implies existence of the asymptotic complexity constant and obtain a sharp lower bound thereby answering two conjectures by Anema \cite{anema}. 
	}
\end{abstract}

\maketitle

\section{Introduction.}\noindent
An important research topic in graph theory is the enumeration of spanning trees. In recent works, various authors have studied the number of spanning trees on so-called fractal graphs, which are graphs approximating a finitely ramified self-similar set (cf.\ \cite{CCY07,CS06,CW06,SW00,TeWa06}). 

In \cite{anema}, a different methodology from the previous works is given for calculating the number of spanning trees on those fractal graphs based on Kirchhoff's  Matrix--Tree Theorem and the specific knowledge of the spectrum of the probabilistic graph Laplacian on such fractals. The method works due to the so-called spectral decimation property studied in various papers such as \cite{MR2450694},\cite{FS92}, \cite{Sh96} among others. Moreover, in \cite{anema} the asymptotic complexity constant of such graphs is studied, which is defined as the limit of the logarithm of the number of spanning trees over the number of vertices of the graphs. In \cite{TeWa11} the asymptotic complexity constant is studied among others, and precisely evaluated in a closed form for a more general class of self similar fractal graphs satisfying a symmetry condition called strong symmetry. Here we study it initially without any symmetry conditions and later with a symmetry condition which is called full symmetry. Our approach is based on the spectral decimation property thus viewing the subject from a point of view which is closer to analysis on fractals. It is conjectured in \cite{anema} that by adding the assumption of full symmetry we can obtain a proof of the existence of the asymptotic complexity constant using the knowledge of the spectrum of the probabilistic graph Laplacian and based on that, a second conjecture is presented regarding a sharp lower bound of the asymptotic complexity constant. The goal of this paper is to prove these two conjectures (Conjectures $4.3$ and $4.4$ in \cite{anema}) and also obtain an upper bound of the asymptotic complexity constant. Thus we obtain the following result.
 
\begin{thm}
For a given fully symmetric self-similar structure on a finitely ramified fractal $K$, let $G_n$ denote its sequence of approximating graphs. Then the asymptotic complexity constant exists and if $G_0$ has more than two vertices we have that
$$\frac{\log(3)}{2} \leqslant c	_{asymp} \leqslant \log{\left(\frac{(m-1)|V_0|(|V_0|-1)}{|V_1|-|V_0|}\right)}$$
where $|V_0|,|V_1|$ are the number of vertices of the $G_0$ and $G_1$ graphs respectively, and $m$ is the number of contractions that create the self similar fractal graph.

\end{thm}

\begin{remark} 
	If we consider the $m$-Tree fractal in \cite{FS10} we have by Cayley's formula that $\tau(G_0)= m^{m-2}$ and thus we see that $\tau(G_n)=m^{(m-2)m^n}$ and $|V_n|=1+(m-1)m^n$ and thus the asymptotic complexity constant is  $\frac{(m-2)\log{m}}{m-1}$. By considering the $3$-Tree fractal for $m=3$ we observe that the asymptotic complexity constant is $\frac {\log{3}}{2}$ which means that the lower bound is sharp.
\end{remark}

\section{Background notions.} \noindent

First, we present some background material. The discussion here is brief, we refer the reader to \cite{MR2450694,anema, Ki01, St06} for more details. If we have a compact connected metric space $(X,d)$ and $F_i: X \rightarrow X$ are injective contractions for $i=1,2,...\ m,$ then there exists a unique non-empty compact subset $K$ of $X$ that satisfies 
$$K=F_1(K)\cup \cdot \cdot \cdot \cup F_m(K).$$
and $K$ is called the self-similar set with respect to  $\{F_1,F_2,...F_m\}$. If for any two distinct words $w,w' \in W_n=\{1, \dots , m\}^n$ we have that $F_w(K) \cap F_{w'}(K) = F_w(V_0) \cap F_{w'}(V_0)$ then we call $K$ a finitely ramified self similar set. Informally, a finitely ramified self-similar set is such that every cell can be made disconnected with the rest of the set by removing a finite number of points.

For any self-similar finitely ramified set $K$ with respect to $\{F_1,F_2,...F_m\}$, we define the sequence of approximating graphs $G_n$ with vertex set $V_n$ in the following way. For all $n\geq0$ and for all 
$\omega\in W_n$ we have that $G_{0}$ is defined as the complete graph with vertices $V_{0}$ and 
$$V_n:= \bigcup_{i=1}^m F_i(V_{n-1})=\bigcup_{w \in W_n}V_w \;\; \; \text{ and }\; \; \; G_{n}:= \bigcup_{w \in W_n}G_w,$$
where $G_w $ is the complete graph with vertices $V_w$ and  $F_{\omega}:=F_{w_1}\circ F_{w_{2}}\circ\cdots F_{w_n} $ for $\omega=w_1 w_2\cdots w_n$. We have  $x,y\in V_n$ to be connected with an edge in $G_n$ if $F_i^{-1}(x)$ and $F_i^{-1}(y)$ are connected by an edge in $G_{n-1}$ for some $1\leq i\leq m$.

A fully symmetric finitely ramified self-similar structure is a self-similar structure $K$ with contractions $\{F_1,F_2,...F_m\}$ such that $K$ is a finitely ramified self-similar set, and for any permutation $\sigma:V_0\rightarrow V_0$ there is an isometry $g_{\sigma}:K\rightarrow K$ that maps any $x\in V_0$ into $\sigma(x)$ and preserves the self-similar structure of $K$. Then we have a map $\tilde{g_{\sigma}}:W_1\rightarrow W_1$ such that $F_i\circ g_{\sigma}=g_{\sigma}\circ F_{\tilde{g_{\sigma}}(i)}$ $\forall i\in W_1$. 

Then, if $G_n$ for $n\geq0$ is a sequence of finite graphs, denote $|V_n|$ to be the cardinality of $V_n$, and $\tau(G_n)$ denote the number of spanning trees of $G_n$. If the limit
$$\lim_{n\rightarrow\infty} \frac{\log(\tau(G_n))}{|V_n|}$$
exists, it is called the \emph{asymptotic complexity constant} or the \emph{tree entropy} of the sequence $G_n$. We also have that for any two, finite, connected graphs $G_1$, $G_2$, if $G_1\vee_{x_1,x_2} G_2$ denotes the graph formed by identifying the vertex $x_1\in G_1$ with vertex $x_2\in G_2$ then $\forall x_1\in G_1, x_2\in G_2$, and it is clear that
\begin{equation}{\label{eqn:wedge}}
 \tau(G_1\vee_{x_1,x_2} G_2)=\tau(G_1)\cdot \tau(G_2).
\end{equation}

We denote by $\Delta_n$ the probabilistic graph Laplacian of $G_n$ which is defined as $\Delta_n = I-T^{-1}A$ where $T$ is the degree matrix and $A$ is the adjacency matrix. The way to study the spectrum of $\Delta_n$ is by a process called spectral decimation which recursively computes the spectrum of $\Delta_{n+1}$ by using information from the spectrum of $\Delta_n$. It was first studied rigorously on the Sierpinski Gasket by \cite{FS92} and later it was generalised to other fractals. We refer the reader to  \cite{MR2450694, Baj2}. In \cite{MR2450694} we have the following propositions which describe the essence of spectral decimation.

We start with $V_0$ being the complete graph on the boundary set. Write $\Delta_1$ in block form
$$\Delta_1=\begin{pmatrix}A&B\\
C&D\\
\end{pmatrix}$$ 
where A is a square block matrix associated to the boundary points. Since the $V_1$ network never has an edge joining two boundary points A is the $|V_0|\times|V_0|$ identity matrix. 
The Schur Complement of $\Delta_1$ is 
\begin{equation*}
S(z)=(A-zI)-B(D-z)^{-1}C
\end{equation*}

\begin{prop}For a given fully symmetric finitely ramified self-similar structure K there are unique scalar valued rational functions $\phi(z)$ and $R(z)$ such that for $z\notin\sigma(D)$
	\begin{equation*}
	S(z)=\phi(z)(P_0-R(z))
	\end{equation*}
\end{prop}

Now, we let 
\begin{equation*}
E(\Delta_0,\Delta_1):=\sigma(D)\bigcup\{z:\phi(z)=0\}
\end{equation*} 
and call $E(\Delta_0,\Delta_1)$ the exceptional set. 
\\
We denote  $\text{mult}_D(z)$ the multiplicity of $z$ as an eigenvalue of $D$ and $\text{mult}_n(z)$ the multiplicity of $z$ as an eigenvalue of $\Delta_n$. In the case that it's not an eigenvalue, we simply say it has multiplicity zero. Then we may inductively find the spectrum of $\Delta_n$ with the following proposition. 

\begin{prop} If $K$ is a fully symmetric finitely ramified self-similar structure and $R(z),\ \phi(z),\ E(\Delta_0,\Delta_1)$ as above, the spectrum of $\Delta_n$ may be calculated recursively in the following way:
	\begin{enumerate}
		\item if $z\notin E(\Delta_0,\Delta_1)$, then
		\begin{equation*}
		\text{mult}_n(z)=\text{mult}_{n-1}(R(z))
		\end{equation*}
		\item if $z\notin \sigma(D)$, $\phi(z)=0$ and $R(z)$ has a removable singularity at z then,
		\begin{equation*}
		mult_n(z)=|V_{n-1}|
		\end{equation*}
		\item if $z\in \sigma(D)$, both $\phi(z)$ and $\phi(z)R(z)$ have poles at z, $R(z)$ has a removable singularity at z, and $\frac{\partial}{\partial z}R(z)\neq 0$, then
		\begin{equation*}
		\text{mult}_n(z)=m^{n-1}\text{mult}_D(z)-|V_{n-1}|+\text{mult}_{n-1}(R(z))
		\end{equation*}
		\item if $z\in \sigma(D)$, but $\phi(z)$ and $\phi(z)R(z)$ do not have poles at z, and $\phi(z)\neq 0$,then 
		\begin{equation*}
		\text{mult}_n(z)=m^{n-1}mult_D(z)+mult_{n-1}(R(z))
		\end{equation*}
		\item if $z\in \sigma(D)$, but $\phi(z)$ and $\phi(z)R(z)$ do not have poles at z, and $\phi(z)=0$,then 
		\begin{equation*}
		\text{mult}_n(z)=m^{n-1}\text{mult}_D(z)+|V_{n-1}|+\text{mult}_{n-1}(R(z))
		\end{equation*}
		\item if $z\in \sigma(D)$, both $\phi(z)$ and $\phi(z)R(z)$ have poles at z, $R(z)$ has a removable singularity at z, and $\frac{\partial}{\partial z}R(z)=0$, then
		\begin{equation*}
		\text{mult}_n(z)=m^{n-1}\text{mult}_D(z)-|V_{n-1}|+2\text{mult}_{n-1}(R(z))
		\end{equation*}
		\item if $z\notin \sigma(D)$, $\phi(z)=0$ and $R(z)$ has a pole at z, then $\text{mult}_n(z)=0$.
		\item if $z\in\sigma(D)$, but $\phi(z)$ and $\phi(z)R(z)$ do not have poles at z, $\phi(z)=0$ and $R(z)$ has a pole at z, then 
		\begin{equation*}
		\text{mult}_n(z)=m^{n-1}\text{mult}_D(z).
		\end{equation*}
	\end{enumerate}
\end{prop} 

As in the notation of \cite{anema} it is possible to define sets of eigenvalues $A,B$ and corresponding multiplicities for $\alpha\in A$, $\alpha_n:=\text{mult}_n(\alpha)$ and for $\beta \in B$, $\beta_n^k:=\text{mult}_n(R_{(-k)}(\beta))$.

Then the spectral decimation algorithm means that the spectrum of the probabilistic graph Laplacian of level $n$ is
\begin{equation} \label{eq:spectrum}
\sigma(\Delta_n)\setminus\{0\}=\bigcup_{\alpha\in A} \left\{\alpha\right\} \bigcup_{\beta\in B}\left[\bigcup_{k=0}^{n}\bigl\{R_{-k}(\beta):\beta_n^k\neq0\bigr\}\right].
\end{equation}

A famous theorem in graph theory is Kirchhoff's Matrix-Tree Theorem which relates the number of spanning trees of a graph with the product of the eigenvalues of its Laplacian. A version of this theorem for the probabilistic graph Laplacian is the following.

\begin{thm} [Kirchhoff's theorem] For any connected, loopless graph $G$ with $n$ labelled vertices, the number of spanning trees of $G$ is
$$\tau(G)=\left|\frac{\left(\prod\limits_{j=1}^n d_j\right)}{\left(\sum\limits_{j=1}^n d_j\right)} \prod_{j=1}^{n-1} \lambda_j \right| $$	where $\lambda_j$ are the non-zero eigenvalues of the probabilistic graph Laplacian.
\end{thm}

Kirchhoff's theorem is used in \cite{anema} to compute the number of spanning trees of the approximating graphs of a given fully symmetric self-similar structure on a finitely ramified fractal $K$. Let $R(z)=\frac{P(z)}{Q(z)}$ with $deg(P(z))> deg(Q(z))$,  $d$ be the degree of $R(z)$, $P_d$ be the leading coefficient of the numerator of $R(z)$, $|V_n|$ be the number of vertices of $G_n$ and $d_j$ be the degree of vertex $j$ in $G_n$. Then the number of spanning trees is given by 
\begin{equation} \label{eq:1}
\tau(G_n)=\left|\frac{\left(\prod_{j=1}^{|V_n|} d_j \right)}{\left(\sum_{j=1}^{|V_n|}d_j \right)}\left(\prod_{\alpha\in A}\alpha^{\alpha_n}\right)\left[\prod_{\beta\in B}\left(\beta^{\sum_ {k=0}^{n}{\beta_n^k}}\left(\frac{-Q(0)}{P_d}\right)^{\sum_{k=0}^n\beta_n^k\left(\frac{d^k-1}{d-1}\right)}\right)\right]\right|.
\end{equation}
We will also make use of the Stolz-Ces\`{a}ro lemma which states the following.

\begin{lem}
	
Let $(a_n)_n$ and $(b_n)_n$ be sequences of real numbers such that $(b_n)_n$ is strictly monotone and divergent to $+ \infty$ or $-\infty$. If we have that the following limit exists 
$$\lim_{n \rightarrow \infty} \frac{a_{n+1}-a_n}{b_{n+1}-b_n}=c,$$ then we have that  	$\lim_{n \rightarrow \infty} \frac{a_n}{b_n}=c$
\end{lem}

Now, we are ready to present the proof of the main theorem.
\section{Proof of the conjectures.}\noindent 
\noindent

As before, let $G_n$ be the sequence of approximating graphs, $|V_n|$ the number of its vertices, $m$ the number of $0$-cells of the $G_1$ graph and $\tau(G_n)$ be the number of spanning trees. Due to the self-similarity we have that the $G_n$ graph is $m^n$ copies of the $G_0$ graph with an appropriate identification of its vertices. In fact, in \cite{TeWa11} we have the formula $|V_n|=m|V_{n-1}|-m|V_0|+|V_1|$ from which we can derive that 
$$|V_n|=\frac{m^n(|V_1|-|V_0|)+m|V_0|-|V_1|}{m-1}.$$ Thus we can see that $\lim_{n \rightarrow \infty} \frac{|V_n|}{m^n}=\frac{|V_1|-|V_0|}{m-1}$. To simplify notation later, we denote this as $|V_n| \sim m^n$.

We give a proof as to why the asymptotic complexity constant is positive and give an upper bound. 

\begin{prop}
Let $K$ be a self similar, finitely ramified fractal and $G_n$ be its sequence of approximating graphs. If $G_1$ is not a tree we have that
$$0<\liminf_{n \rightarrow \infty} \frac{\log{\tau(G_n)}}{|V_n|} \leqslant \limsup_{n \rightarrow \infty} \frac{\log{\tau(G_n)}}{|V_n|}	\leqslant \log{\left(\frac{(m-1)|V_0|(|V_0|-1)}{|V_1|-|V_0|}\right)}$$
	
\end{prop}

\begin{remark}
The first inequality, that it is greater than zero, is proven in \cite{anema}, we repeat the argument here for the convenience of the reader. We make one modification however. The author asserts that if we assume that $G_1$ is not a tree then it implies that $|V_0|>2$. However, this is not the case. There are self similar fractal graphs such that $G_0$ is the complete graph on two vertices and $G_1$ is not a tree. Examples of that are the graphs studied in \cite{2frac} and the so called Austria graphs in \cite{TeWa11}. It is interesting then to see what happens in the case that $|V_0|=2$. Then we may have that $G_1$ is a tree and thus the number of spanning trees is trivially $1$ and therefore the asymptotic complexity constant is $0$. However if we assume that $G_1$ is not a tree, then we will see that a very similar argument to the previous case proves that even if $|V_0|=2$ we have a positive lower bound.  
\end{remark}

\begin{proof}
We have that  $\tau(G_n)\geq \tau(G_0\vee_{x,x}^{m^n}G_0)$ where $G_0\vee_{x,x}^{m^n}G_0$ denotes $m^n$ copies of $G_0$ each identified to each other at some vertex $x\in V_0$. Assume first that $|V_0|>2$. Then, since the $G_0$ graph is the complete graph on $|V_0|$ vertices, by Cayley's formula we have that $\tau(G_0)=|V_0|^{(|V_0|-2)}$. Thus we get that $$\tau(G_0\vee_{x,x}^{m^n}G_0)=|V_0|^{(|V_0|-2)\cdot m^n}$$ and 
$$ \tau(G_n)\geq |V_0|^{(|V_0|-2)\cdot m^n}.$$ So for $n\geq0$, 
\begin{equation}\label{eqn:lowerasy}
\log(\tau(G_n))\geq m^n\cdot (|V_0|-2)\log(|V_0|)
\end{equation} 
which is equation (6) in \cite{anema}. This gives us the result since $|V_n| \sim m^n$. In the case that $|V_0|=2$ we can say the following. Since $G_1$ is not a tree, it must be that $m>2$ and $G_1$ must contain a cycle and of course $|V_1|>2$. Then using the same argument as before, we have that $G_n$ is $m^{n-1}$ copies of $G_1$ with appropriate vertex identification and   $\tau(G_n)\geq \tau(G_1\vee_{x,x}^{m^{n-1}}G_1)$ for some vertex $x \in V_1$. Since $G_1$ contains a cycle, it must have at minimum $3$ spanning trees which gives us that $ \tau(G_n)\geq 3^{m^{n-1}}$ and we obtain the lower bound exactly as before.

Now, for the upper bound. First we observe that if we denote $EV_n$ the cardinality of the edge set of $G_n$ then we have that $EV_n=\frac{m^n |V_0| (|V_0|-1)}{2}$. This can be seen from the self similarity of the graph and the fact that $G_0$ is the complete graph on $V_0$ vertices. Also, we have from Kirchhoff's theorem that 
$$ \log{\tau(G_n)}=\log{\frac{\left(\prod\limits_{j=1}^{|V_n|} d_j\right)}{\left(\sum\limits_{j=1}^{|V_n|} d_j\right)}}+\log{ \prod_{j=1}^{V_n-1} \lambda_j}  $$
The first summand becomes $\sum\limits_{j=1}^{|V_n|} \log{d_j}-\log{\sum\limits_{j=1}^{|V_n|} d_j}$ and by using Jensen's inequality, we obtain that
\begin{equation*}
\begin{split}
\sum\limits_{j=1}^{|V_n|} \log{d_j}-\log{\sum\limits_{j=1}^{|V_n|} d_j} &\leqslant \sum\limits_{j=1}^{|V_n|} \log{d_j} \leqslant |V_n| \log{\left(\frac{\sum\limits_{j=1}^{|V_n|} d_j}{|V_n|}\right)}\\ 
&=|V_n| \log{\frac{2EV_n}{|V_n|}}=|V_n| \log{\frac{m^n |V_0|(|V_0|-1)}{|V_n|}}. 
\end{split}
\end{equation*}
Since $\lim_{n \rightarrow \infty} \frac{|V_n|}{m^n}=\frac{|V_1|-|V_0|}{m-1}$ we get an upper bound for $|V_n|^{-1}\log{\frac{\left(\prod\limits_{j=1}^{|V_n|} d_j\right)}{\left(\sum\limits_{j=1}^{|V_n|} d_j\right)}}$.
Now for the term $\log{ \prod_{j=1}^{|V_n|-1} \lambda_j} $, we know that the trace of the probabilistic graph Laplacian matrix equals $|V_n|$ and therefore as before
\begin{equation*}
\begin{split}
\sum\limits_{j=1}^{|V_n|-1} \log{\lambda_j}& \leqslant (|V_n|-1) \log{\left(\frac{\sum\limits_{j=1}^{|V_n|-1} \lambda_j}{|V_n|-1}\right)}=(|V_n|-1) \log{\frac{|V_n|}{|V_n|-1}}\\ 
&= \log{\left(\frac{|V_n|}{|V_n|-1}\right)^{|V_n|-1}}\rightarrow \log{e}=1. 
\end{split}
\end{equation*}
Thus $|V_n|^{-1}\log{ \prod_{j=1}^{|V_n|-1} \lambda_j} \leqslant 0$ which concludes our proof.
\end{proof}

We now add the extra assumption of full symmetry and prove the main theorem.
\begin{proof}

We want to prove the existence of the limit of the sequence $\frac{\log{\tau(G_n)}}{|V_n|}$. We already have from the proposition above that the sequence is bounded. Therefore it suffices to check that we do not have any oscillatory behavior. By the full symmetry assumption, we can perform spectral decimation and $\tau(G_n)$ is given by equation (\ref*{eq:1}) and thus we obtain that 

\begin{equation}
\begin{split}
\log{\tau(G_n)}=&\log{\left| \frac{\prod\limits_{j=1}^{|V_n|} d_j}{\sum\limits_{j=1}^{|V_n|}d_j} \right| }+ \log{\left|\prod_{\alpha\in A}\alpha^{\alpha_n}\right|}\\
& +\log{\left|\prod_{\beta\in B} \beta^{\sum_{k=0}^n\beta_n^k}\right|}+\log{\left|\prod_{\beta\in B}\left(\frac{-Q(0)}{P_d}\right)^{\sum_{k=0}^n\beta_n^k\frac{d^k-1}{d-1}}\right|}
\end{split}
\end{equation}
Therefore it suffices to prove that the limit 
$\lim\limits_{n \rightarrow \infty} \frac{\log{\prod\limits_{j=1}^{|V_n|} d_j-\log{\sum\limits_{j=1}^{|V_n|}d_j}}}{|V_n|}$ exists and for each $\alpha \in A$ and $\beta \in B$ the limits 
$$\lim_{n \rightarrow \infty} \frac{\alpha_n}{|V_n|}, \, \lim_{n \rightarrow \infty} \frac{\sum_{k=0}^n\beta_n^k}{|V_n|} \, \text{ and } \, \lim_{n \rightarrow \infty} \frac{\sum_{k=0}^n\beta_n^k \frac{d^k-1}{d-1}}{|V_n|}$$
also exist.
We know that even though $\Delta_n$ is not in general a symmetric operator it has only real eigenvalues. Then, including multiplicities, their total number must be equal to the dimension of the space and thus from equation (\ref*{eq:spectrum}) we have the formula
$$\sum_{\alpha\in A}{\alpha_n} + \sum_{\beta \in B}\sum_{k=0}^n\beta_n^k d^k +1=|V_n| .$$
Since $\alpha_n$ , $\beta_n^k$ are non-negative integers we see that for each $\alpha \in A $ and $\beta \in B$ that $ \frac{\alpha_n}{|V_n|}$, $\frac{\sum_{k=0}^n\beta_n^k d^k}{|V_n|}$ must be bounded and thus the same holds for $\frac{\sum_{k=0}^n\beta_n^k}{|V_n|}$ and  $\frac{\sum_{k=0}^n\beta_n^k \frac{d^k-1}{d-1}}{|V_n|}$.
Now, for a given $\alpha \in A$, we have that by the definition of the finite set $A$ that the multiplicities $\alpha_n=\text{mult}_n(\alpha)$ which can be found from Proposition 1.3. above depend only on the eigenvalue $\alpha$ and the level $n$ and in each of the cases of the Proposition we have convergence as $|V_n| \sim m^n$.
Now for the remaining limits. Take $\beta \in B$ and $\beta_n^k=\text{mult}_n(R_{-k}(\beta))$. By the general algorithm of the spectral decimation methodology, we have that every pre-iterate of the spectral decimation rational function preserves the multiplicity of the eigenvalues. Therefore, we have that $\beta_{n+1}^k=\beta_n^{k-1}$ for $1 \leq k \leq n+1$ and thus the sum of multiplicities at level $n+1$ must be the sum of the multiplicities at level $n$ along with those with generation of birth $n+1$. This is just the following formula 
$$\sum_{k=0}^{n+1}\beta_{n+1}^k=  \sum_{k=1}^{n+1}\beta_{n+1}^k+\beta_{n+1}^0 =\sum_{k=1}^{n+1}\beta_{n}^{k-1} +\beta_{n+1}^0 =\sum_{k=0}^{n}\beta_{n}^k+\beta_{n+1}^0$$ and
$$\sum_{k=0}^{n+1}\beta_{n+1}^k d^k=  \sum_{k=1}^{n+1}\beta_{n}^k d^k+\beta_{n+1}^0 =\sum_{k=1}^{n+1}\beta_{n}^{k-1} d^k +\beta_{n+1}^0 =d \sum_{k=0}^{n}\beta_{n}^k d^k+\beta_{n+1}^0$$ By taking into account that $\frac{|V_{n+1}|}{|V_{n}|} \rightarrow m$ and by looking at the Proposition 1.3. above, we have a list of possible choices for the term $\beta_{n+1}^0$ and as similarly to the case of the eigenvalues in the set $A$ before it must be that $\frac{\beta_{n+1}^0}{V_{n+1}}$ converges to a finite positive constant, which we can call $c$.

\begin{comment}
From the arguments above, we have that $V_{n+1}=y_n V_n$ where $y_n$ is a sequence such that $y_n \rightarrow m$. We obtain then that
$$ \frac{\sum_{k=0}^{n+1}\beta_{n+1}^k-\sum_{k=0}^{n}\beta_{n}^k}{V_{n+1}-V_n}=\frac{\sum_{k=0}^{n}\beta_{n}^k+\beta_{n+1}^0-\sum_{k=0}^{n}\beta_{n}^k}{ynV_{n}-V_n}=\frac{\beta_{n+1}^0}{V_n(y_n-1)} \rightarrow \frac{mc}{m-1} $$
 Thus by Stolz-Ces\`{a}ro theorem we get the existence of  $\lim_{n \rightarrow \infty} \frac{\sum_{k=0}^n\beta_n^k}{|V_n|}$.
 
\end{comment}

For a general first order linear recurrence $S_{n+1}=f_n S_n+g_n$ we know that it has solution
$$S_n=\left( \prod_{k=0}^{n-1}f_k\right) \left(A + \sum_{m=0}^{n-1} \frac{g_m}{\prod_{k=0}^m f_k} \right)$$ where $A$ is a constant.
From the arguments above, we have that $V_{n+1}=y_n V_n$ where $y_n$ is a sequence such that $y_n \rightarrow m$ and $\frac{\beta_{n+1}^0}{V_{n+1}}=c+x_n$ with $x_n$ being a sequence such that $x_n \rightarrow 0$. Then for $S_n= \frac{\sum_{k=0}^n\beta_n^k d^k}{|V_n|} $ we obtain that 
$S_{n+1}=\frac{d}{y_n}S_n+c+x_n$. Since we know that $S_n$ is bounded, it must be that $\frac{d}{y_n} \leq 1-\epsilon$ for some $\epsilon >0 $ and large $n$. Then
$$S_n=\left( \prod_{k=0}^{n-1}\frac{d}{y_k}\right) \left(A + \sum_{i=0}^{n-1} \frac{c+x_i}{\prod_{k=0}^i \frac{d}{y_k}} \right)$$ 
We care about the limit of $n\rightarrow \infty$ so the constant part becomes $0$ and we are left with 
$$c d^n \frac{\sum_{i=0}^{n-1} \prod_{k=0}^i \frac{y_k}{d}}{\prod_{k=0}^{n-1}y_k}  + d^n \frac{\sum_{i=0}^{n-1} x_i \prod_{k=0}^i \frac{y_k}{d}}{ \prod_{k=0}^{n-1}y_k}  $$
The second summand goes to $0$ as can be seen by the Stolz-Ces\`{a}ro lemma in the following way. Due to the fact that  $\frac{d}{y_n} \leq 1-\epsilon$  we have that $ \prod_{k=0}^{n-1}\frac{y_k}{d}$ is a strictly increasing sequence diverging to $+\infty$. Then,
$$\frac{\sum_{i=0}^{n} x_i \prod_{k=0}^i \frac{y_k}{d}-\sum_{i=0}^{n-1} x_i \prod_{k=0}^i \frac{y_k}{d}}{\prod_{k=0}^n \frac{y_k}{d}-\prod_{k=0}^{n-1} \frac{y_k}{d}}=\frac{x_n \prod_{k=0}^n \frac{y_k}{d}}{\prod_{k=0}^{n-1} \frac{y_k}{d}(\frac{y_n}{d}-1)} \rightarrow 0$$
since $y_n \rightarrow m$ and $x_n \rightarrow 0$.

The first summand is just $\sum_{i=0}^{n-1} d^{n-i+1} \prod_{k=i+1}^{n-1} \frac{1}{y_k}$ which is a positive series and since $S_n$ is bounded, it must be that it converges. Thus we get existence of  $\lim_{n \rightarrow \infty} \frac{\sum_{k=0}^n\beta_n^k d^k}{|V_n|}$.
By an exact similar argument, or more easily by the Stolz-Ces\`{a}ro lemma, we have the existence of the limit $\lim_{n \rightarrow \infty} \frac{\sum_{k=0}^n\beta_n^k }{|V_n|}$ and thus also we get that $\lim_{n \rightarrow \infty} \frac{\sum_{k=0}^n\beta_n^k \frac{d^k-1}{d-1}}{|V_n|}$ exists.

We have that $ |V_n| ^ {-1} \log{\left| \frac{\prod\limits_{j=1}^{|V_n|} d_j}{\sum\limits_{j=1}^{|V_n|}d_j} \right| }$ is bounded and that $\lim_{n \rightarrow \infty}\frac{\log{\sum\limits_{j=1}^{|V_n|}d_j}}{|V_n|}=0$. Moreover the limit $\lim_{n \rightarrow \infty}\frac{\log{\prod\limits_{j=1}^{|V_n|} d_j}}{|V_n|}$ cannot oscillate due to the symmetry of the fractal graph and thus obviously exists as it's bounded.
Thus all the required limits exist and we obtain our result.

Now, to bound the limit from below. Regarding the number of vertices, we have the following bound
$$|V_n| \leqslant m^n (|V_0|-1)+1.$$
This follows due to the fact that the $G_n$ graph is $m^n$ copies of the $G_0$ one and therefore we obviously have that $|V_n| \leqslant m^n |V_0|$. However, due to connectivity, some vertices need to overlap. At minimum, one vertex from each $0$-cell will overlap which would mean that 
$$|V_n| \leqslant m^n |V_0|-1-1-...-1$$
with the number of $-1$ being as many times as the cells minus one, namely $m^n-1$ which would give us $ |V_n| \leqslant m^n |V_0| -m^n +1$.
  
\noindent Then we have the following,
$$\frac{1}{|V_n|}\geq \frac{1}{ m^n (|V_0|-1)+1}=\frac{1}{ m^n\left(|V_0|-1+\frac{1}{m^n}\right)}$$
Then by the above inequality (2.1), we get that
$$\frac{\log(\tau(G_n))}{|V_n|} \geq \frac {m^n (|V_0|-2) \log(|V_0|)}{m^n\left(|V_0|-1+\frac{1}{m^n}\right)}$$
and thus 
$$ \lim_{n \rightarrow \infty} \frac{\log(\tau(G_n))}{|V_n|} \geq \frac{(V_0-2)\log|V_0|}{|V_0|-1}.$$
However, since $|V_0|$ is an integer strictly greater than two and we can define the function $f:[3,+\infty] \rightarrow \mathbb{R} , \; f(x)= \frac{(x-2)\log{x}}{x-1}$ and observe that it has a global minimum at $x=3$ and therefore 
$$ \lim_{n \rightarrow \infty} \frac{\log(\tau(G_n))}{|V_n|} \geq \frac{\log3}{2}.$$
Thus the asymptotic complexity constant must be at least $\frac{\log3}{2}$ which concludes the proof.

\end{proof}
  
Lastly, we illustrate with an example how this methodology can be applied by evaluating the number of spanning trees on the level $3$ Sierpinski Gasket. This has been previously evaluated for example in \cite{CCY07}, \cite{TeWa11} but we reevaluate it here in a different way using the methodology of \cite{anema}. For a variety of different examples we refer the reader to \cite{anema}.

\begin{prop}
The number of spanning trees on the level $n$ graph approximation of $SG_3$ is given for $n \geq 2$ by
$$\tau(G_n)=2^{a_n}3^{b_n}5^{c_n}7^{d_n}$$
where $a_n=\frac{2}{5}(6^n-1)$, $b_n=\frac{1}{25}(13 \cdot 6^n-15n+12)$, $c_n=\frac{1}{25}(3 \cdot 6^n-15n-3)$, $d_n= \frac{1}{25}(7\cdot 6^n+15n-7)$.
\end{prop}
\begin{proof}
To calculate the term $\frac{\prod_{j=1}^{|V_n|} d_j}{\sum_{j=1}^{|V_n|}d_j}$ it suffices to observe that in $SG_3$ the boundary points have degree $2$, the central ones in each cell have degree $6$ and all the rest have degree $4$. Moreover, we know that $|V_n|= 3+ \frac{7}{5}(6^n-1)$. Thus this gives us that
$$ \frac{\prod_{j=1}^{|V_n|} d_j }{\sum_{j=1}^{|V_n|}d_j }    = \frac{2^3 \cdot 4 ^{\frac{7}{5}(6^n-1)-\sum_{k=1}^n 6^{n-k}} \cdot 6^{\sum_{k=1}^n 6^{n-k}} }{4[3+\frac{7}{5}(6^n-1)]-6+2\cdot \sum_{k=1}^n 6^{n-k}} = 2^{-\frac{3}{5}+\frac{13}{5}6^n-n} \cdot 3^{\frac{6^n-6}{5}-n}$$
The spectrum of $\Delta_n$ has been evaluated in \cite{Baj2} using spectral decimation. The rational function of the spectral decimation algorithm is $R(z)= \frac{6z(z-1)(4z-5)(4z-3)}{6z-7}$ and thus $d=4$, $P_d=6\cdot 4^2$ and $Q(0)=-7$. We also have that $A= \{ \frac{3}{2} \} $ and $B= \{ 1, \frac{3}{4} , \frac{5}{4} , \frac{3+ \sqrt{2} }{4} , \frac{3- \sqrt{2}}{4}  \}$. For the corresponding multiplicities we have that for $\alpha =\frac{3}{2}$, $\alpha_n = \frac{2 \cdot 6^n +8}{5}$.  For $\beta_n^k =1$ , $\beta_n^k =1$. For $\beta = \frac{3}{4}$ or $\beta =\frac{5}{4}$ , $\beta_n^k = \frac{3}{5}(6^{n-k-1}-1) , \beta_n^{n-1}=\beta_n^n=0$. For $ \beta = \frac{3+\sqrt{2}}{4}$ or $\beta = \frac{3-\sqrt{2}}{4}$ we have that $\beta_n^k =\frac{2\cdot 6^{n-k-1}+8}{5} , \beta_n^n =0$. The eigenvalue $1$ has multiplicity one. Then
\begin{equation}\label{eqn:sierpbeta}
\begin{split}
\prod_{\alpha\in A} \alpha^{\alpha_n} & \prod_{\beta\in B}\left(\beta^{\sum_{k=0}^n\beta_n^k}\cdot\left(\frac{7}{4^2 \cdot 6}\right)^{\sum_{k=0}^n\beta_n^k\frac{4^k-1}{3}}\right)= \left(\frac{3}{2}\right)^{\displaystyle \frac{2 \cdot 6^n+8}{5}}\times\\
&\times \left(\frac{3}{4}\right)^{\displaystyle \sum_{k=0}^{n-2}\frac{3}{5}(6^{n-k-1}-1)}\times\left(\frac{7}{4^2 \cdot 6}\right)^{\displaystyle \sum_{k=0}^{n-2}\frac{3}{5}(6^{n-k-1}-1) \frac{4^k-1}{3}}\\
&\times\left(\frac{5}{4}\right)^{\displaystyle \sum_{k=0}^{n-2}\frac{3}{5}(6^{n-k-1}-1)}\times\left(\frac{7}{4^2 \cdot 6}\right)^{\displaystyle \sum_{k=0}^{n-2}\frac{3}{5}(6^{n-k-1}-1) \frac{4^k-1}{3}}\\
&\times\left(\frac{3+\sqrt{2}}{4}\right)^{\displaystyle \sum_{k=0}^{n-1}\frac{2\cdot 6^{n-k-1}+8}{5}}\times\left(\frac{7}{4^2 \cdot 6}\right)^{\displaystyle \sum_{k=0}^{n-1}\frac{2\cdot 6^{n-k-1}+8}{5} \cdot \frac{4^k-1}{3}}\\
&\times\left(\frac{3-\sqrt{2}}{4}\right)^{\displaystyle \sum_{k=0}^{n-1}\frac{2\cdot 6^{n-k-1}+8}{5}}\times\left(\frac{7}{4^2 \cdot 6}\right)^{\displaystyle \sum_{k=0}^{n-1}\frac{2\cdot 6^{n-k-1}+8}{5} \cdot \frac{4^k-1}{3}}\\
&\times 1 \cdot \left( \frac{7}{4^2 \cdot 6} \right) ^{\sum_{k=0}^{n-1}\frac{4^{k}-1}{3}}\\
\end{split}
\end{equation}
By calculating the sums above we obtain that
\begin{align*}
&\sum_{k=0}^{n-2}\frac{3}{5}(6^{n-k-1}-1)=\frac{1}{25}\left(3\cdot 6^n-15n-3\right)\\
&\sum_{k=0}^{n-2}\frac{3}{5}(6^{n-k-1}-1) \frac{4^k-1}{3}=\frac{1}{5}\left(\frac{9}{5}\cdot 6^{n-1}-\frac{10}{3} \cdot 4^{n-1}+n+\frac{8}{15}\right)\\
&\sum_{k=0}^{n-1}\frac{2 \cdot 6^{n-k-1}+8}{5}=\frac{2}{25}\left(6^n+20n-1\right)\\
&\sum_{k=0}^{n-1}\left(\frac{2 \cdot 6^{n-k-1}+8}{5}\right)\frac{4^k-1}{3}=\frac{6^n-4^n}{15} + \frac{8 (4^n-1)}{45} - \frac{2}{75}\left(6^n+20n-1\right)\\
&\sum_{k=0}^{n-1}\frac{4^{k}-1}{3}=\frac{4^n-1}{9}-\frac{n}{3}.
\end{align*}
Then by combining those equations above, and doing some elementary calculations we obtain our result.
\end{proof}

\begin{remark} 
The asymptotic complexity constant is  $ C_{asympt} =\frac{2}{7} \log2 +\frac{13}{35}\log3 + \frac{3}{35} \log5 + \frac{1}{5} \log7$. 
\end{remark}

\section*{Acknowledgements}

We are grateful to Alexander Teplyaev, Robert S. Strichartz, Anders Karlsson and Anders \"Oberg, for valuable suggestions.


\begin{thebibliography}{9}
	
\bibitem{anema}
Jason A. Anema
\newblock {\em Counting Spanning Trees on Fractal Graphs}
\newblock arXiv preprint arXiv:1211.7341 (2012), Cornell PhD Thesis. 



\bibitem{MR2450694}
N.~Bajorin, T.~Chen, A.~Dagan, C.~Emmons, M.~Hussein, M.~Khalil, P.~Mody,
B.~Steinhurst, and A.~Teplyaev.
\newblock Vibration modes of {$3n$}-gaskets and other fractals.
\newblock {\em J. Phys. A}, 41(1):015101, 21, 2008.


\bibitem{Baj2}
N.~Bajorin, T.~Chen, A.~Dagan, C.~Emmons, M.~Hussein, M.~Khalil, P.~Mody,
B.~Steinhurst, and A.~Teplyaev.
\newblock Vibration spectra of finitely ramified, symmetric fractals.
\newblock {\em Fractals} 16.03 : 243-258, 2008.


\bibitem{CCY07}
Shu-Chiuan Chang, Lung-Chi Chen, and Wei-Shih Yang.
\newblock Spanning trees on the {S}ierpinski gasket.
\newblock {\em J. Stat. Phys.}, 126(3):649--667, 2007.

\bibitem{CS06}
Shu-Chiuan Chang and Robert Shrock.
\newblock Some exact results for spanning trees on lattices.
\newblock {\em J. Phys. A}, 39(20):5653--5658, 2006.

\bibitem{CW06}
Shu-Chiuan Chang and Wenya Wang.
\newblock Spanning trees on lattices and integral identities.
\newblock {\em J. Phys. A}, 39(33):10263--10275, 2006.


\bibitem{FS10}
Daniel~J. Ford and Benjamin Steinhurst.
\newblock Vibration spectra of the {$m$}-tree fractal.
\newblock {\em Fractals}, 18(2):157--169, 2010.


\bibitem{FS92}
M.~Fukushima and T.~Shima.
\newblock On a spectral analysis for the {S}ierpi\'nski gasket.
\newblock {\em Potential Anal.}, 1(1):1--35, 1992.

\bibitem{functional eq}
Gregory Derfel , Peter J. Grabner and Fritz Vogl.
\newblock Laplace operators on fractals and related functional equations.
\newblock {\em J. Phys. A} 45 (2012), no. 46, 463001.



\bibitem{HST11}
Katheryn Hare, Benjamin Steinhurst, Alexander Teplyaev, and Denglin Zhou.
\newblock Disconnected julia sets and gaps in the spectrum of laplacians on
symmetric finitely ramified fractals.
\newblock {\tt arXiv:1105.1747v2}, 2011.


\bibitem{Ki01}
Jun Kigami.
\newblock {\em Analysis on fractals}, volume 143 of {\em Cambridge Tracts in
	Mathematics}.
\newblock Cambridge University Press, Cambridge, 2001.


\bibitem{Ly05}
Russell Lyons.
\newblock Asymptotic enumeration of spanning trees.
\newblock {\em Combin. Probab. Comput.}, 14(4):491--522, 2005.

\bibitem{Ly10}
Russell Lyons.
\newblock Identities and inequalities for tree entropy.
\newblock {\em Combin. Probab. Comput.}, 19(2):303--313, 2010.

\bibitem{allodentro}
Leonid Malozemov
\newblock Random walk and chaos of the spectrum. Solvable model.
\newblock {\em Chaos Solitons Fractals}, 5 (1995), no. 6, 895–907.

\bibitem{MR1997913}
Leonid Malozemov and Alexander Teplyaev.
\newblock Self-similarity, operators and dynamics.
\newblock {\em Math. Phys. Anal. Geom.}, 6(3):201--218, 2003.


\bibitem{2frac}
Leonid Malozemov and Alexander Teplyaev.
\newblock Pure point spectrum of the Laplacians on fractal graphs.
\newblock  {\em J. Funct. Anal.},  129  (1995),  no. 2, 390--405.


\bibitem{Sh96}
Tadashi Shima.
\newblock On eigenvalue problems for {L}aplacians on p.c.f. self-similar sets.
\newblock {\em Japan J. Indust. Appl. Math.}, 13(1):1--23, 1996.

\bibitem{SW00}
Robert Shrock and F.~Y. Wu.
\newblock Spanning trees on graphs and lattices in {$d$} dimensions.
\newblock {\em J. Phys. A}, 33(21):3881--3902, 2000.

\bibitem{St06}
Robert~S. Strichartz.
\newblock {\em Differential equations on fractals}.
\newblock Princeton University Press, Princeton, NJ, 2006.
\newblock A tutorial.



\bibitem{TeWa06}
Elmar Teufl and Stephan Wagner.
\newblock The number of spanning trees of finite sierpi\'{n}ski graphs.
\newblock {\em In: Fourth Colloquium on Mathematics and Computer Science},
Nancy:411--414, 2006.

\bibitem{TeWa07}
Elmar Teufl and Stephan Wagner.
\newblock Enumeration problems for classes of self-similar graphs.
\newblock {\em J. Combin. Theory Ser. A}, 114(7):1254--1277, 2007.


\bibitem{TeWa12}
Elmar Teufl and Stephan Wagner.
\newblock Resistance scaling and the number of spanning trees in self-similar lattices.
\newblock {\em Journal of Statistical Physics} 142, no. 4 (2011): 879-897.

\bibitem{TeWa11}
Elmar Teufl and Stephan Wagner.
\newblock The number of spanning trees in self-similar graphs.
\newblock {\em Ann. Comb.}, 15(2):355--380, 2011.

\end{thebibliography}
\end{document}